\documentclass[12pt]{article}

\usepackage{amsmath}
\usepackage{amsfonts}
\usepackage{xcolor}
\usepackage{amssymb}
\usepackage{graphicx,bm}
\usepackage{enumerate}
\usepackage{amsthm,amscd}
\usepackage[all]{xy}

 \allowdisplaybreaks

\usepackage{hyperref}

\newtheorem{theorem}{Theorem}[section]
\newtheorem{corollary}[theorem]{Corollary}
\newtheorem{definition}[theorem]{Definition}
\newtheorem{conjecture}[theorem]{Conjecture}

\newtheorem{lemma}[theorem]{Lemma}

\newtheorem{proposition}[theorem]{Proposition}

\begin{document}

\title{Maximum size of a triangle-free graph with bounded maximum degree and matching number\thanks{M. Ahajideh and T. Ekim were supported by TÜBİTAK grant number 118F397.
M.A.~Yıldız was supported by a Marie Skłodowska-Curie Action from the EC (COFUND grant no. 945045) and by the NWO Gravitation project NETWORKS (grant no. 024.002.003).\\
 e-mail addresses: ahanjidm@gmail.com (M. Ahanjideh),  
tinaz.ekim@boun.edu.tr(T. Ekim), m.a.yildiz@uva.nl (M. A. Yıldız).}}

\author{
Milad Ahanjideh$^{\dag}$, Tınaz Ekim$^{\dag}$ and Mehmet Akif Yıldız $^{\ddagger}$\\
{\small$^{\dag}$ Department of Industrial Engineering, Bo\u{g}azi\c{c}i University,} \\ {\small 34342, Bebek, Istanbul,Turkey.} \\
{\small $^{\ddagger}$ Korteweg de Vries Instituut voor Wiskunde, Universiteit van Amsterdam,}\\ {\small Postbus 94248
1090 GE Amsterdam, The Netherlands. }}

\maketitle

\vspace{-0.5cm}

\begin{abstract}
Determining the maximum number of edges under degree and matching number constraints have been solved for general graphs in \cite{hanson} and \cite{j}. It follows from the structure of those extremal graphs that deciding whether this maximum number decreases or not when restricted to claw-free graphs, to $C_4$-free graphs or to triangle-free graphs are separately interesting research questions. The first two cases being already settled in \cite{Pinar} and \cite{chordal}, in this paper we focus on triangle-free graphs. We show that unlike most cases for claw-free graphs and $C_4$-free graphs, forbidding triangles from extremal graphs causes a strict decrease in the number of edges and adds to the hardness of the problem. We provide a formula giving the maximum number of edges in a triangle-free graph with degree at most $d$ and matching number at most $m$ for all cases where $d\geq m$, and for the cases where $d<m$ with either $d\leq 6$ or $Z(d)\leq m < 2d$ where $Z(d)$ is a function of $d$ which is roughly $5d/4$. We also provide an integer programming formulation for the remaining cases and as a result of further discussion on this formulation, we conjecture that our formula giving the size of triangle-free extremal graphs is also valid for these open cases.
 \end{abstract}



\section{Introduction} 
In extremal graph theory, an important series of problems, including the celebrated Turán's graphs \cite{turan}, investigate the maximization or the minimization of the number of edges in a graph under a given set of constraints. A question of this kind is to determine the maximum number of edges of a graph when its maximum degree is at most $d$ and its matching number is at most $m$ for two given integers $d$ and $m$. This is a special case of a more general problem posed by Erdős and Rado in 1960 \cite{erdos-rado}. It is worth mentioning that this problem is equivalent to determining Ramsey numbers for line graphs \cite{ramseyline}. This question has been first solved in 1974 by Chv\'{a}tal and Hanson \cite{hanson} using some optimization techniques. A proof constructing an ``extremal'' graph with maximum number of edges under given degree and matching number constraints has only came out much later in 2009 by Balachandran and Khare \cite{j}. Balachandran and Khare \cite{j} exhibit an extremal graph whose connected components consist of stars, complete graphs and in some cases ``almost complete graphs'' that contain $C_4$'s (cycles of length 4), but do not inform us on the unicity of these extremal graphs. This gives rise to a natural question: what happens if we restrict the structure of extremal graphs? Can the same upper bound be still achieved? The structure of extremal graphs given in \cite{j} makes this question especially interesting for three classes of graphs obtained by restricting the above-mentioned types of components: claw-free graphs obtained by forbidding the smallest star (which is not an edge), triangle-free graphs obtained by forbidding the smallest complete graph (which is not an edge), and $C_4$-free graphs (since $C_4$'s occur in ``almost complete graphs''). 

Among these directions, the situation of claw-free graphs has been settled by Dibek et al. in \cite{Pinar}. The authors exhibit cases where the maximum number of edges remains the same as for general graphs, and other cases where it is strictly less. More recently, Blair et al. \cite{chordal} investigated chordal graphs which are much more restricted than $C_4$-free graphs, the class of graphs that would exclude the ``almost complete graph'' components occurring in the extremal graphs provided in \cite{j}. The authors showed that replacing the ``almost complete graph'' components by chordal graphs having the same size, the bound for general graphs is also achieved by chordal graphs.  In the same spirit, M{\aa}land addressed the restriction to bipartite graphs, split graphs, disjoint unions of split graphs and unit interval graphs in \cite{thesie}. 

In this paper, we investigate the direction that remained open and consider triangle-free graphs from the same perspective. We start with same preliminaries in Section \ref{sec:prem}. In Section \ref{sec:dgeqm}, we first determine the maximum number of edges of a triangle-free graph when its maximum degree is at most $d$ and its matching number is at most $m$ for two given integers $d$ and $m$ such that $d>m$ or $d=m$. Besides, for $m>d$, we derive some structural properties for the connected components of an edge-extremal graph, which allows us to identify the desired extremal value in further sections. Using these structural properties, in Section \ref{sec:d456}, we solve the problem for  $m>d$ with either $d\leq 6$ or $Z(d)\leq m <2d$ where $Z(d)$ is roughly $5d/4$. For claw-free graphs and chordal graphs, the size of edge-extremal graphs are the same as the general upper bound in most of the cases. Clearly, this guarantees the optimality of the size once a graph with desired properties is constructed. Unlike these cases, the size of edge-extremal triangle-free graphs that we find in this paper is, in most of the time, strictly less than the general case. This adds to the difficulty of proving the optimality in our results. In Section \ref{sec:conc}, we present all our findings as a unique formula providing the size of the extremal graphs (in Theorem \ref{thm:ALL-IN-ONE}) and compare it with the size of general extremal graphs. Last but not least, in Section \ref{sec:discussion}, we investigate the remaining cases, namely for natural numbers $m$ and $d$ such that $7\leq d<m$ with either $m<Z(d)$ or $m\geq 2d$. For these open cases, we suggest an integer programming formulation based on our earlier observations. With further discussion on this formulation, we conjecture that the formula we provide in Theorem \ref{thm:ALL-IN-ONE} is valid in general, with no condition on $d$ and $m$. Lastly, again based on our former structural results, we reformulate our problem as a variant of the extremal problem addressed in Turan's Theorem \cite{turan} with an additional constraint on the maximum degree; or in Erdős-Stone's Theorem which has been described as a fundamental theorem of extremal graph theory (see \cite{bollobas}). Indeed, the problem of finding the maximum number of edges in a $K_r$-free graph with given number of vertices and maximum degree at most $d$ is an interesting problem for itself. 

\section{Notation and Preliminaries}\label{sec:prem}

Throughout  this paper, $G=(V(G),E(G))$ is a simple undirected graph. We call $|V(G)|$ and $|E(G)|$  the {\it order} and the {\it size} of $G$, respectively. For any vertex $v\in V(G)$, the number of vertices adjacent to $v$ is said to be the \textit{degree} of $v$, denoted by $d(v)$. We say a graph $G$ is \textit{$d$-regular} if $d(v)=d$ for all $v\in V(G)$. Moreover, if $d(w)=d-1$ for some $w\in V(G)$ and $d(v)=d$ for all $v\in V(G)-w$, then $G$ is said to be \textit{almost $d$-regular}. We denote the maximum degree of $G$ by $\Delta(G)$, and the minimum degree of $G$ by $\delta(G)$. The minimum number of colors to color all edges of a graph $G$ in such a way that two adjacent edges receive different colors is called the \textit{chromatic index} of $G$, and denoted by $\chi'(G)$. According to Vizing's Theorem, we have $\Delta(G)\leq \chi'(G)\leq \Delta(G)+1$ for any graph $G$ \cite{vizing}. Given a graph, a set of edges having pairwise no common end vertex is called a \textit{matching}. The size of a maximum matching of $G$ is called \textit{the matching number} and denoted by $\nu(G)$. We say that $G$ has a \textit{perfect matching} if $\nu(G)=n/2$,  where $n=|V(G)|$.
The complete graph of order $n$ and the complete bipartite graph with sets of sizes $m$ and $n$ are denoted by $K_n$ and $K_{m,n}$, respectively. The graph $K_{1,d}$ is called  a \textit{$d$-star}. A graph is \textit{triangle-free} if it does not contain $K_3$ as an induced subgraph. 

For a given graph class $\mathbf{C}$ and two given positive integers $d$ and $m$, we define $\mathbb{M}_\mathbf{C}(d,m)$ to be the set of all graphs $G$ in $\mathbf{C}$ satisfying $\Delta(G)\leq d$ and $\nu(G)\leq m$. A graph in $\mathbb{M}_\mathbf{C} (d,m)$ with the maximum number of edges is called \textit{edge-extremal}, and the number of edges of an edge-extremal graph in $\mathbb{M}_\mathbf{C} (d,m)$ is denoted by $f_{\mathbf{C}}(d,m)$. Let $\vartriangle$ be the class of triangle-free graphs. In this paper, we assume that edge-extremal graphs have no isolated vertices since adding isolated vertices to a graph does
not increase the number of edges. 

We note that in general, if one of the two parameters  $\Delta(G)$ and $\nu(G)$ is not bounded, then the size of $G$ is not bounded neither (for general graphs). Indeed, a star has matching number one no matter how large its degree, thus its size. Likewise, the graph consisting of an unbounded number of independent $K_2$'s (that is, sharing no common vertex) is an example where the maximum degree is bounded (by one) but the matching number is not, neither the size. It follows from this discussion that, in general, one should bound both the matching number and the degree of a graph so that its size is also bounded. In this case, Vizing's Theorem provides us with a natural upper bound on the size of a graph. For any graph $G$, since the set of edges having the same color in an edge-coloring of $G$ forms a matching whose size is at most $\nu(G)$, and we have $\chi'(G)\leq \Delta(G)+1$ by Vizing's Theorem; we obtain $|E(G)|\leq (\Delta(G)+1)\nu(G)$. For given bounds $\Delta(G)\leq d$ and $\nu(G) \leq m$, an edge-extremal graph can thus have at most $dm+m$ edges. The maximum size of a general graph with $\Delta(G)\leq d$ and $\nu(G) \leq m$ obtained in \cite{hanson} and \cite{j} shows that this upper bound is actually met when some divisibility conditions hold, and we are ``pretty close'' to it otherwise. The following theorem gives not only the formula for the maximum size of a (general) graph with $\Delta(G)\leq d$ and $\nu(G) \leq m$, but also describes an edge-extremal graph. Let $\mathcal{GEN}$ denote the class of general graphs. 

\begin{theorem}[\hspace{-0.01mm}\cite{j}]
\label{thm:general}
With the preceding notation, we have, \[f_{\mathcal{GEN}}(d,m)=dm + \left\lfloor \frac{d}{2} \right\rfloor  \left\lfloor \frac{m}{\lceil \frac{d}{2} \rceil} \right\rfloor. \]

Moreover, a graph with $f_{\mathcal{GEN}}(d,m)$ edges is obtained by taking the disjoint union of $r$ copies of $d$-star and $q$ copies of 
\[ \begin{cases} 
      K_{d+1} & $if d+1 is odd$ \\
      K'_{d+1} & $if d+1 is even$,
   \end{cases}\]
where  $q$ is the largest integer such that $m = q \left\lceil \frac{d}{2} \right\rceil + r$ and $r \ge 0$; and where $K'_{d+1}$ is the graph obtained by removing a perfect matching from the complete graph $K_{d+1}$ on $d+1$ vertices, adding a new vertex $v$, and making $v$ adjacent to $d$ of the other vertices. 
\end{theorem}

In this paper, we find the size of triangle-free extremal graphs in most cases; apart from two simple cases, namely $d=1$ and $m<\lfloor d/2\rfloor$, none of them achieves the general upper bound given in Theorem \ref{thm:general}.

Let us now introduce a key lemma that describes the structure of edge-extremal graphs.  A graph $G$ is said to be \textit{ factor-critical} if $G\setminus v$ has a perfect matching for all $v\in V(G)$. By definition, being factor-critical for a graph $G$ directly implies that $|V(G)|=2\nu(G)+1$. We will use the following well-known result which is a sufficient condition for a graph to be factor-critical.

\begin{lemma}(Gallai's Lemma, \cite{g})\label{gallai} 
If $G$ is a connected graph such that for all $v\in V(G)$, $\nu(G\setminus v)=\nu(G)$, then $G$ is factor-critical and hence $|V(G)|=2\nu(G)+1$.
\end{lemma}

The following lemma has been first given in \cite{j} for general graphs, and then restated slightly differently in \cite{chordal}. It establishes a connection between edge-extremal graphs
and factor-critical graphs for a wide range of graph classes, including triangle-free graphs. For the sake of completeness, we also provide a short proof. Let us introduce a special class of extremal graphs that will be our main focus in the rest of the paper.

\begin{definition} $\mathcal{G}_{\mathbf{C}}(d,m)$ is the subclass of the set of edge-extremal graphs in $\mathbb{M}_{\mathbf{C}} (d,m)$ which consists of the graphs having maximum number of connected components isomorphic to a $d$-star.
\end{definition}

\begin{lemma}\cite{j,chordal}\label{1}
Let $d,m$ be natural numbers, and let $\mathbf{C}$ be a graph class that is closed under vertex deletion and closed under taking disjoint union with stars. Take a graph $G\in \mathcal{G}_{\mathbf{C}}(d,m)$. Then, every connected component of $G$ that is not a $d$-star is factor-critical.
\end{lemma}
\begin{proof}
Suppose on the contrary that $W$ is a connected component of $G$ which is neither a $d$-star nor factor-critical. By Lemma
	\ref{gallai}, there is a vertex $v$ in $W$ such that $\nu(W\setminus v)<\nu(W)$. Now we construct a new graph $G'$ whose components are the components of $G$ except $W$, $W\setminus v$ and a $d$-star. One can observe that $G'\in \mathbb{M}_\mathbf{C}(d,m)$ and $|E(G')|=|E(G\setminus v)| +d\geq |E(G)|$. So $G'$  is an edge-extremal graph in $\mathbb{M}_\mathbf{C}(d,m)$ with more star components than in $G$, a contradiction with the assumptions on $G$.
\end{proof}




Lastly, we derive a result that will be useful in Section \ref{sec:d456}. Let $\chi(G)$ denote the minimum number of colors needed to color all vertices of $G$ in such a way that two adjacent vertices get different colors.

\begin{lemma}\cite{a}\label{erdos}
	Let $r\geq 3$. For any graph $G$ on $n$ vertices, at most two of the following properties can hold:
	\begin{enumerate}
		\item $G$ does not contain $K_r$ as an induced subgraph,
		\item $\delta(G)> \dfrac{3r-7}{3r-4} n$,
		\item $\chi(G)\geq r$.
		\end{enumerate}
\end{lemma}
The following corollary states that for $r=3$, if properties 1 and 2 of Lemma \ref{erdos} hold, then property 3 is not satisfied.
\begin{corollary}\label{key}
	Any triangle-free graph of order $n$ with minimum  degree greater than $\dfrac{2n}{5}$ is bipartite.
\end{corollary}

\section{Edge-extremal triangle-free graphs with $d\geq m$}\label{sec:dgeqm}
In this section, we find the maximum number of edges in a triangle-free graph with matching number at most $m$ and degree at most $d$ and where $d> m$.  Besides, we also solve the case where $d=m$. Solving these cases allows us to further strengthen our assumption in Lemma \ref{1} on the structure of an edge-extremal triangle-free graph. Stated in Corollary \ref{cor:structure-not-star}, this structural property will play a key role to obtain our main results for $d<m$ in Sections \ref{sec:d456} and \ref{sec:conc}. First, let us bound the number of edges in a factor-critical triangle-free graph in terms of the matching number.

\begin{lemma}\label{lem:factor-critical-triangle-free}
Let $H$ be a factor-critical triangle-free graph. Then, we have $|E(H)|\leq 1+\nu(H)^2$.
\end{lemma}	
\begin{proof}
Since $H$ is factor-critical, we have $|V(H)|=2h+1$ where $h:=\nu(H)$. 
Bipartite graphs are not factor-critical, therefore $H$ has an odd cycle. Let us take the smallest (induced) odd cycle $C_{2s+1}$ in $H$. Notice that $s\geq 2$ since $H$ is triangle-free. Moreover, there is at most $(h-s)^2$ edges within $H-C_{2s+1}$ by Turan's theorem since $H-C_{2s+1}$ has $2h-2s$ vertices. On the other hand, any vertex in $H-C_{2s+1}$ can have at most $s$ neighbors in $C_{2s+1}$ because otherwise there would be a triangle. As a result, we get 
\begin{eqnarray*}
	|E(G)|&\leq& (2s+1)+(h-s)^2+(2h-2s)s\\
	&\leq & (s^2+1)+(h-s)^2+(2h-2s)s=1+h^2,
\end{eqnarray*} 
which completes the proof.
\end{proof}

\noindent By using Lemma \ref{lem:factor-critical-triangle-free}, we can derive the following structural property for the graphs in $\mathcal{G}_{\vartriangle}(d,m)$.
	
\begin{lemma}\label{lem:notstar}
Let $G\in \mathcal{G}_{\vartriangle}(d,m)$. Then, for any connected component $H$ of $G$ that is not a $d$-star, we have
\begin{itemize}
\item[(i)] $|E(H)|\leq 1+\nu(H)^2$, and
\item[(ii)] $\nu(H)\geq d$.
\end{itemize}
\end{lemma}

\begin{proof}
Since $H$ is not a star, by Lemma \ref{1}, $H$ is factor-critical. Then, part (i) follows from Lemma \ref{lem:factor-critical-triangle-free}.
Now, suppose $\nu(H)<d$. Since $H$ is triangle-free and non-bipartite, we have $\nu(H)\geq 2$. Thus, we get $|E(H)|\leq 1+\nu(H)^2<d\cdot \nu(H)$. Then, take $\nu(H)$ copies of $d$-stars instead of $H$; this increases the number of edges while keeping $\nu(G)$ and $\Delta(G)$ the same. This contradicts with $G\in \mathcal{G}_{\vartriangle}(d,m)$. Therefore, we get $\nu(H)\geq d$, so the result follows.
\end{proof}

\noindent Lemma \ref{lem:notstar} allows us to answer the cases $d > m\geq 1$ (in Theorem \ref{thm:case-d-greater-m}) and $d=m$   (in Theorem \ref{thm:case-d-d}).

\begin{theorem}\label{thm:case-d-greater-m}
With the preceding notation, $f_{\vartriangle}(d,m)=dm$ for $d>m\geq 1$.
\end{theorem}

\begin{proof}
Assume $d>m$, and take $G\in\mathcal{G}_{\vartriangle}(d,m)$. If $G$ has a component $G_1$ that is not a $d$-star, then by Lemma \ref{lem:notstar} (ii), we would get $m\geq\nu(G_1)\geq d$, which is a contradiction. Hence, all the components of $G$ are $d$-stars, so we get $|E(G)|=dm$.
\end{proof}

\begin{theorem}\label{thm:case-d-d}
With the preceding notation, $f_{\vartriangle}(1,1)=1$ and $f_{\vartriangle}(d,d)=d^2+1$ for $d\geq2$.
\end{theorem}

\begin{proof}
Firstly, any graph $G$ with $\Delta(G)=\nu(G)=1$ can contain only one edge, so $f_{\vartriangle}(1,1)=1$ follows. Now, consider the graph $A_d$ shown in Figure \ref{fig:A_d}. It can be easily seen that $A_d\in \mathbb{M}_{\vartriangle}(d,d)$ and $|E(A_d)|=d^2+1$. Then, let us take $G\in \mathcal{G}_{\vartriangle}(d,m)$. By definition of $G$, we have $|E(G)|\geq |E(A_d)|$ giving $|E(G)|\geq d^2+1$. If all the components of $G$ are $d$-stars, then we would get $|E(G)|\leq d^2$, which is a contradiction. Hence, $G$ has at least one component which is not a $d$-star; let us denote it by $G_1$. By Lemma \ref{lem:notstar} (ii), we have $\nu(G_1)\geq d$. Since $d=\nu(G)\geq \nu(G_1)$ we obtain $G_1=G$. Now, by Lemma \ref{lem:notstar} (i), we have $|E(G)|\leq d^2+1$, which completes the proof.

\end{proof}

\begin{figure}[h!]
 \centering
	\includegraphics[scale=0.2]{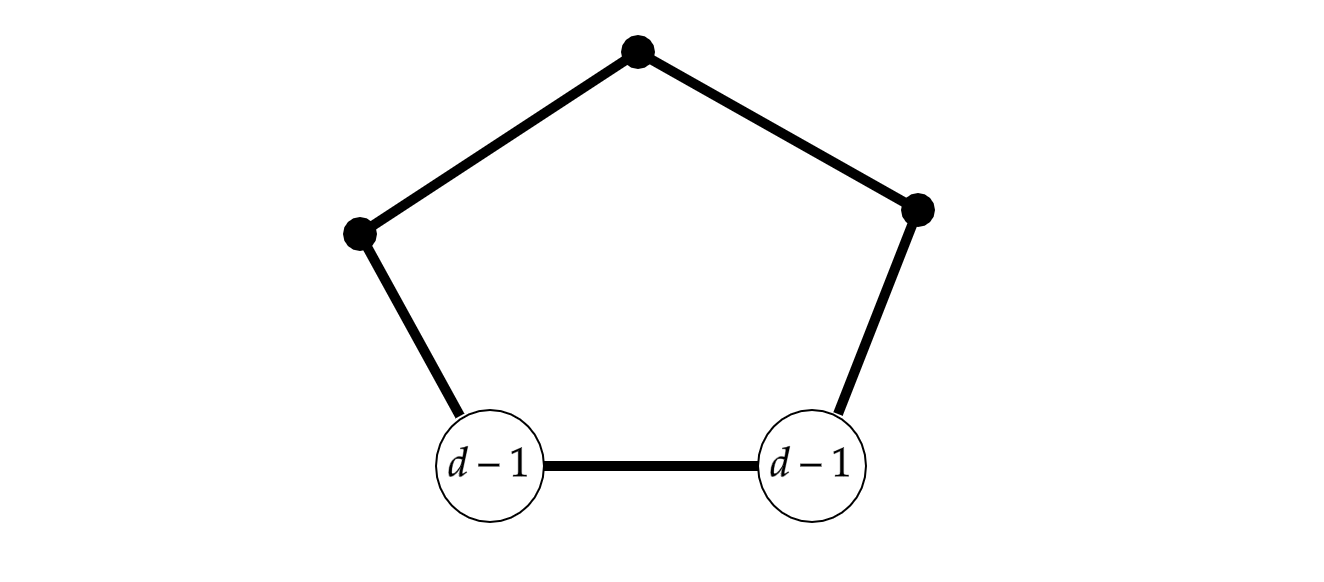}
   \caption{\small{$A_d$ is a graph on $2d+1$ vertices which is a blow-up of a cycle of length five. A circle and the number inside it represent an independent set of that size, and straight lines between two circles or between a vertex and a circle indicate that all possible edges are present. } }
   	\label{fig:A_d}
\end{figure}

We close this section with a corollary of Lemmas \ref{1} and \ref{lem:notstar}, which states that for any edge-extremal graph in $\mathcal{G}_{\vartriangle}(d,m)$ (whose number of $d$-star components is maximum), every component $H$ of it which is not a $d$-star is a factor-critical and edge-extremal graph in  $\mathbb{M}_{\vartriangle}(d,\nu(H))$ with matching number $\nu(H)\geq d$. An extremal graph with these properties will be useful to prove our results in Section \ref{sec:d456}.

	\begin{corollary}\label{cor:structure-not-star}
		Let $d$ and $m$ be natural numbers, and let $G\in \mathcal{G}_{\vartriangle}(d,m)$. Then, for every connected component $H$ of $G$, one of the following is true:
		\begin{itemize}
			\item[(i)] $H$ is a $d$-star.
			\item[(ii)] $|E(H)|=f_{\vartriangle}(d,\nu(H))$ and $|V(H)|=2\cdot \nu(H)+1$ where $\nu(H)\geq d$.
		\end{itemize}
	\end{corollary}
	\begin{proof}
		Let $H$ be a connected component of $G$ that is not a $d$-star. First of all, we know $\nu(H)\geq d$ by Lemma \ref{lem:notstar}. Also, from Lemma \ref{1}, we know that $H$ is factor-critical, thus $|V(H)|=2\cdot \nu(H)+1$. Hence, we have $H\in \mathbb{M}_{\vartriangle}(d,\nu(H))$ since $\Delta(H)\leq d$, which implies $|E(H)|\leq f_{\vartriangle}(d,\nu(H))$. On the other hand, if $|E(H)|<f_{\vartriangle}(d,\nu(H))$, we would get $|E(G)|<|E(G_1)|$ by taking $G_1$ as the disjoint union of $G-H$ and $H_1$ for some $H_1\in\mathbb{M}_{\vartriangle}(d,\nu(H))$, which leads to a contradiction. As a result, we get $|E(H)|=f_{\vartriangle}(d,\nu(H))$. 
	\end{proof}


\section{Edge-extremal triangle-free graphs with $m> d$} \label{sec:d456}

\noindent We start this section with the trivial case $d=1$. Then, we will investigate a deeper study on the structure of extremal graphs to settle two cases with $m>d$, namely $Z(d)\leq m <2d$ for some function $Z(d)$ introduced in Definition \ref{def:minimum-order}, and $d\leq 6$.
\begin{theorem}\label{thm:case-d-1}
With the preceding notation, we have $f_{\vartriangle}(1,m)=m$ for all $m\geq 1$.
\end{theorem}
\begin{proof}
If $\Delta(G)=1$ for a graph $G$, then $G$ is the disjoint union of $\nu(G)$ edges, so the result follows.
\end{proof}

\noindent In the rest of this section, we assume $d\geq 2$. Our results will be based on the following key property. We will show that if $H$ is a connected component of a graph $G\in \mathcal{G}_{\vartriangle}(d,m)$ and if it is not a $d$-star then in addition to the assumption $\nu(H)\geq d$ given in Corollary \ref{cor:structure-not-star} ii), we can also bound $\nu(H)$ from above by $Z(d)$ (see Lemma \ref{lem:bound-5d-over-4}) where $Z(d)$ is defined below and described in Lemma \ref{lem:lower-bound-Z(d)}.  

\begin{definition}\label{def:minimum-order}
For any $d\geq2$, let $Z(d)$ be the smallest natural number $n$ such that there exists a $d$-regular (if $d$ is even) or almost $d$-regular (if $d$ is odd) triangle-free and factor-critical graph $G$ with $\nu(G)=n$.
\end{definition}

\noindent  Let us introduce the graph $B_d$ given in Figure \ref{fig:BLOW-UP}; it is a (almost) $d$-regular triangle-free and factor-critical graph, which shows the existence of $Z(d)$. The \textit{blow-up} of a graph is obtained by replacing every vertex with a finite collection of copies so that the copies of two vertices are adjacent if and only if the originals are. In particular, the copies of the same vertex form an independent set in the blow-up graph. Let us emphasize some properties of the graph $B_d$ in Proposition \ref{prop:properties-B(d)}.

\begin{figure}[h!]
 \centering
		\includegraphics[width=0.7cm,angle=0,height=5.65cm]{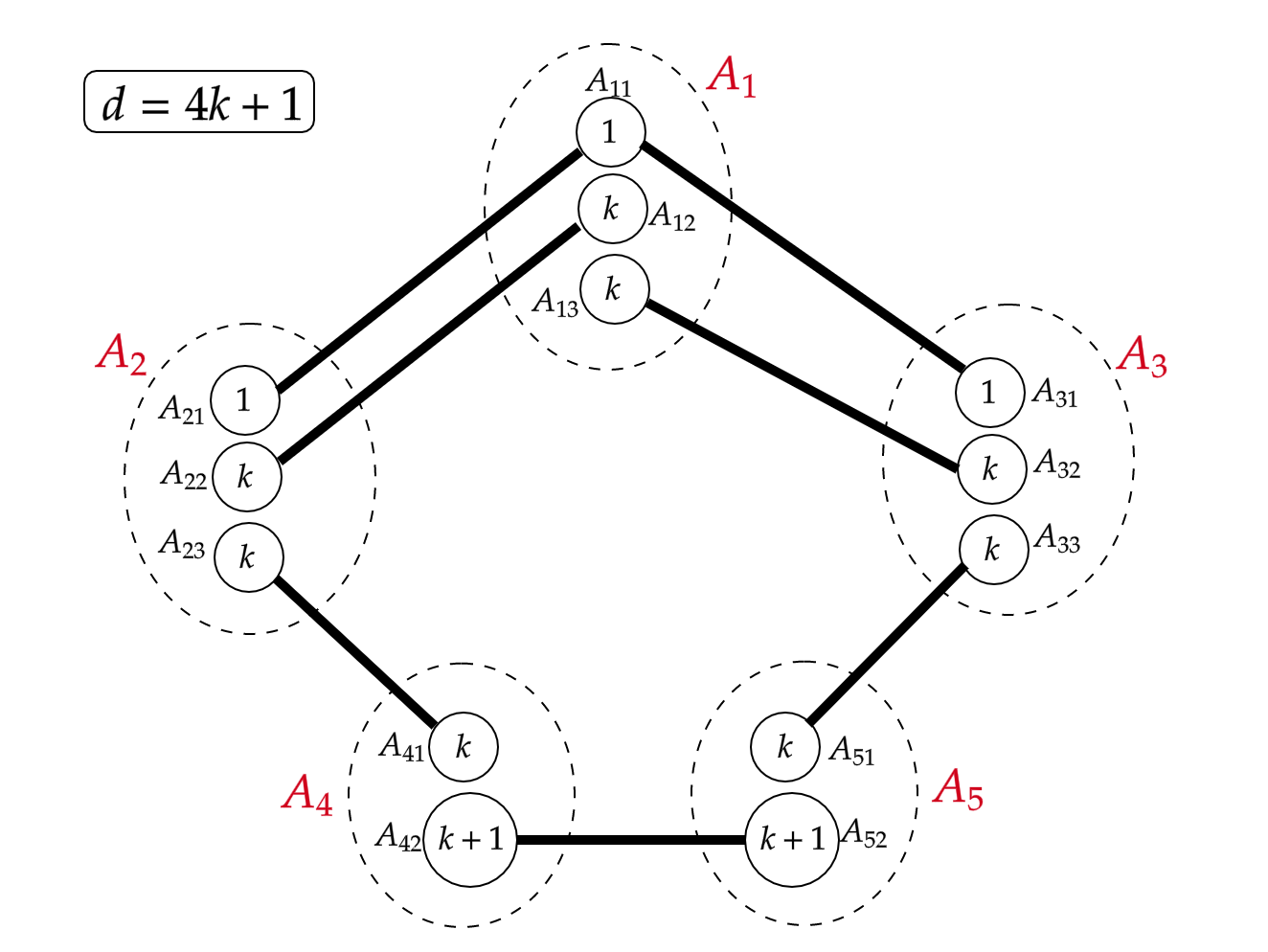}
		\hfill
		\includegraphics[width=0.7cm,angle=0,height=5.65cm]{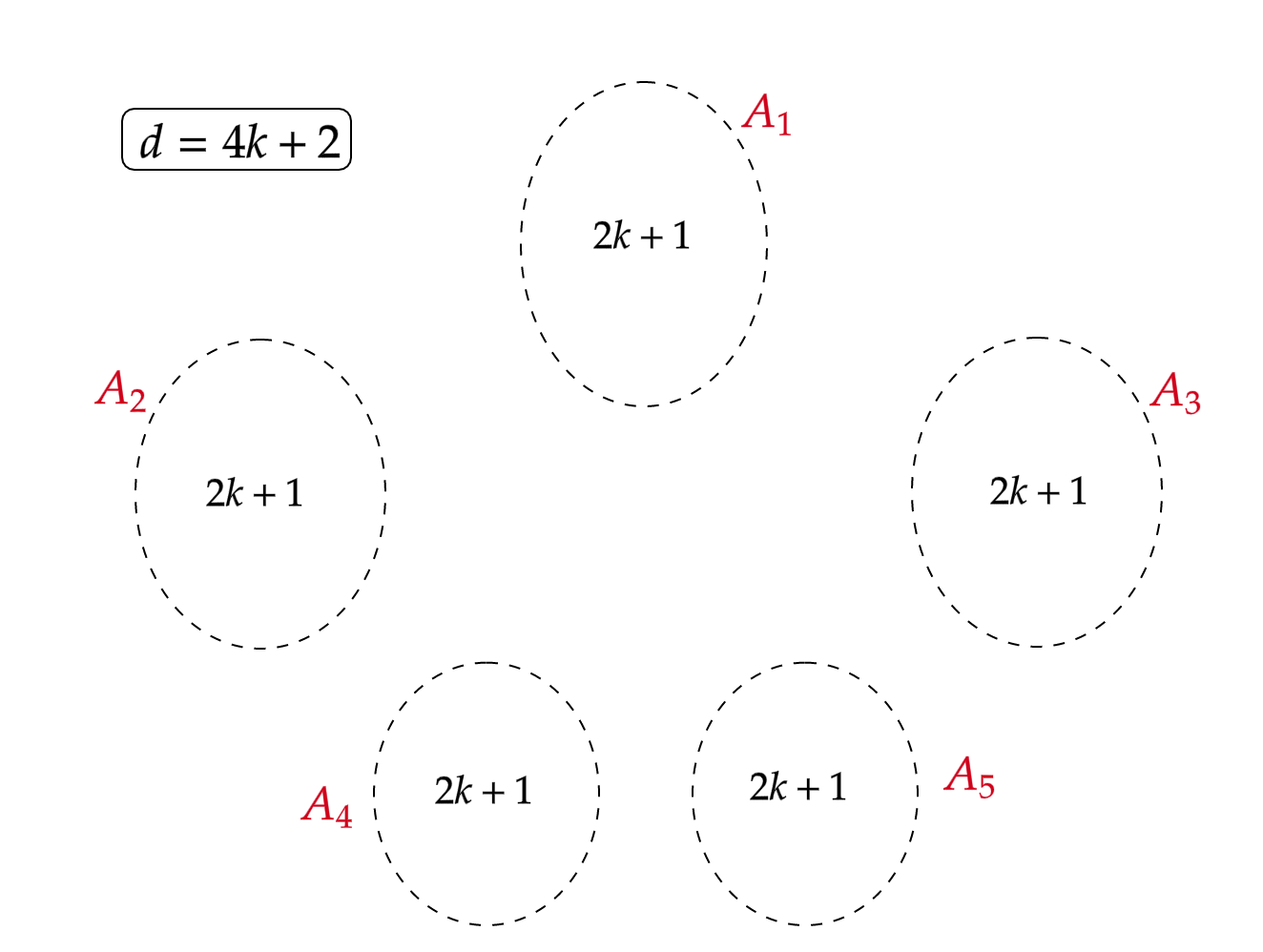}
	\hfill
		\includegraphics[width=0.7cm,angle=0,height=5.65cm]{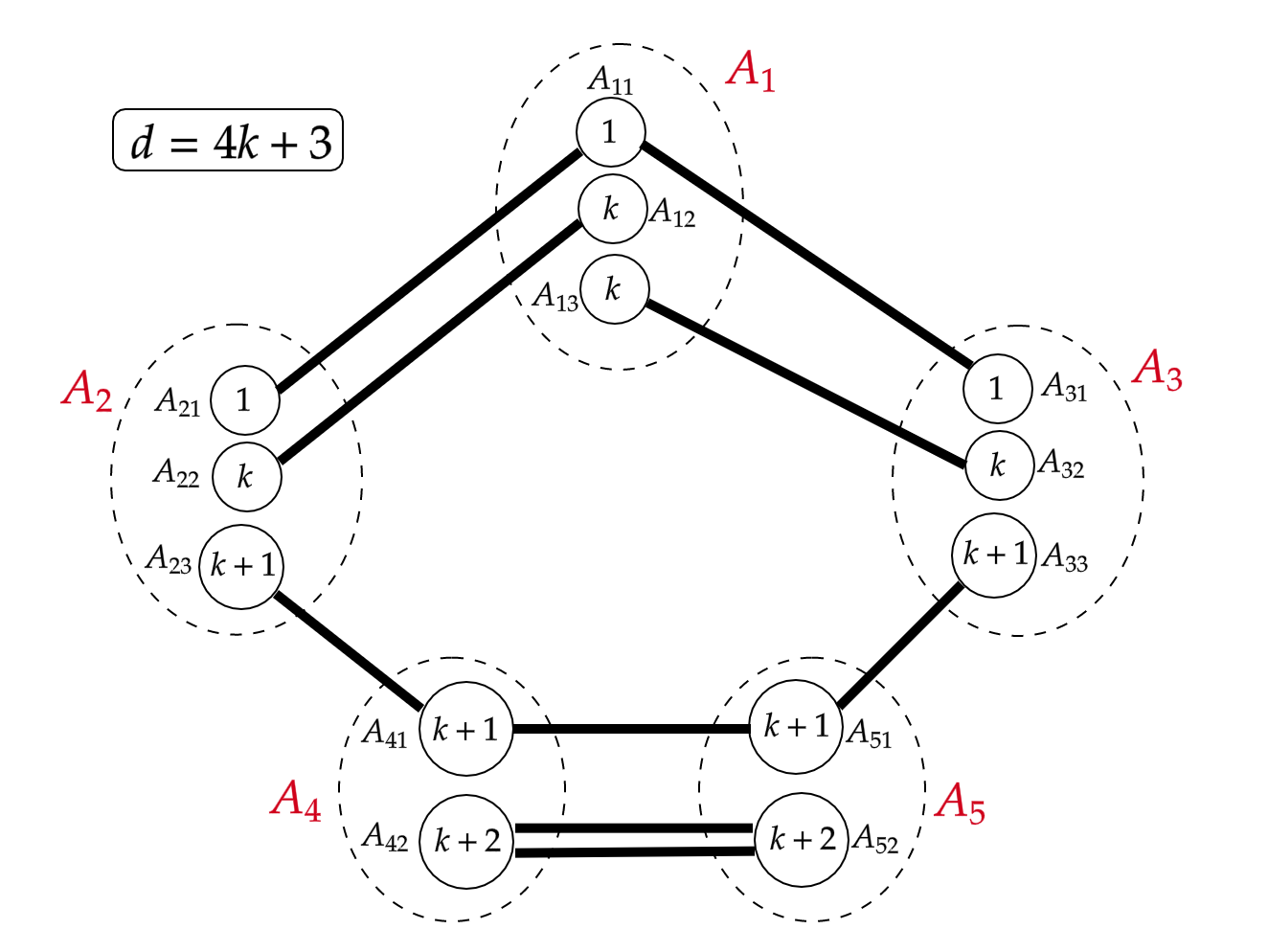}
		\hfill
		\includegraphics[width=0.7cm,angle=0,height=5.65cm]{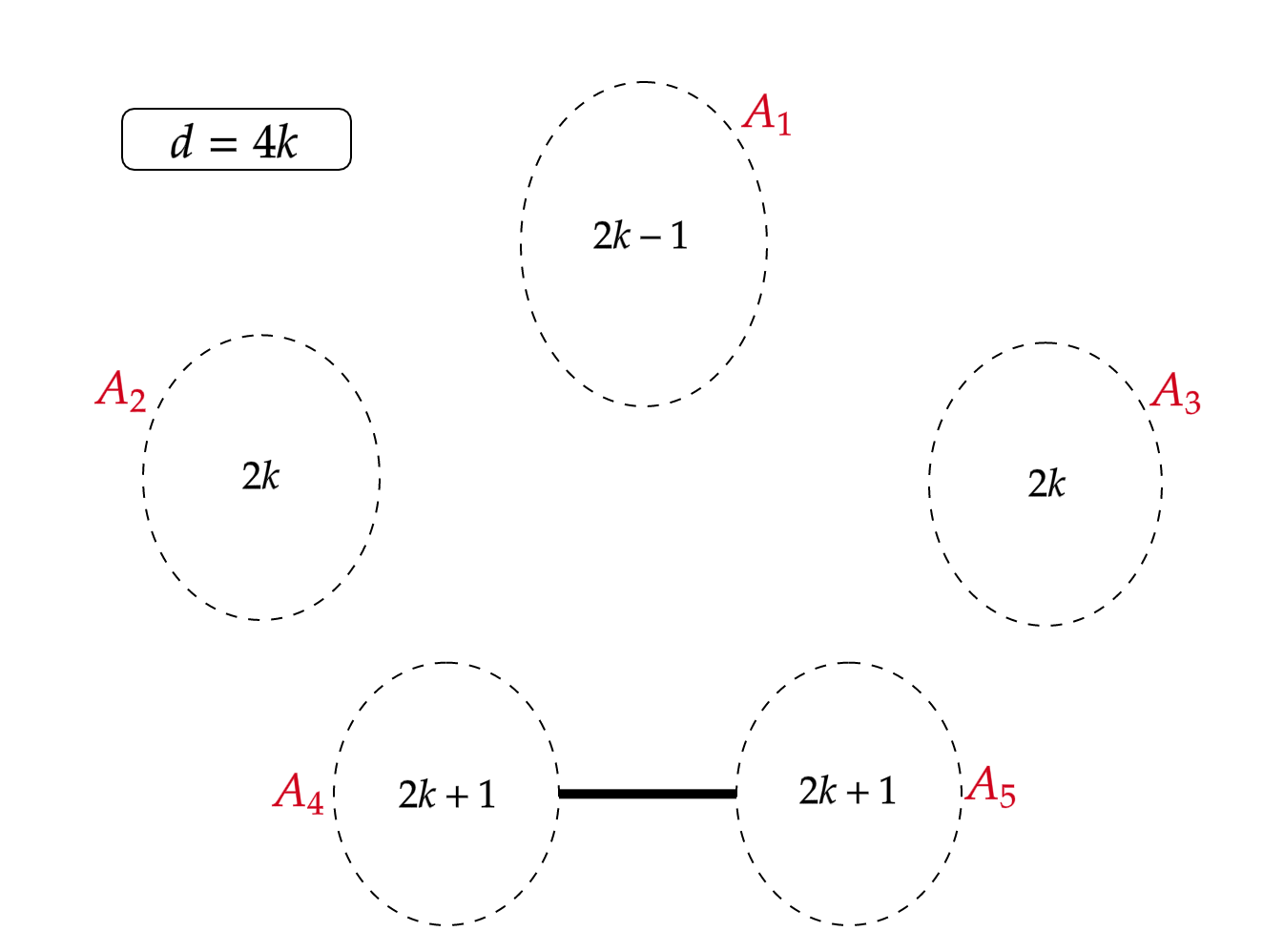}
	\caption{\small{ The graph $B_d$ for $d\geq2$ depending on $d \pmod{4}$. Each graph is obtained from a blow-up of a cycle of length 5 by removing some perfect matchings. For simplicity, edges of the blow-up graph are not shown although they are all present. The copies of the same vertex in the blow-up graph are divided into bags shown by dotted or continuous circles, each one containing as many copies as the number indicated in it.  The lines between different bags represent perfect matchings between the corresponding sets of vertices that are removed from the graph.} }
		\label{fig:BLOW-UP}
\end{figure}

\begin{proposition}\label{prop:properties-B(d)}
The graph $B_d$ in Figure \ref{fig:BLOW-UP} is (almost) $d$-regular, triangle-free, and factor-critical. Moreover, we have  $|V(B_d)|=2\nu(B_d)+1$ and $|E(B_d)|=d \nu(B_d)+\lfloor d/2\rfloor$ where $$\nu(B_d)=\begin{cases}
\lfloor5d/4\rfloor,&\text{ if }d\text{ is even,}\\
\lfloor5(d+1)/4\rfloor,&\text{ if }d\text{ is odd.}
\end{cases}$$  
\end{proposition}
\begin{proof}
Firstly, it can be easily checked that $B_d$ is $d$-regular when $d$ is even. For odd values of $d$, all the vertices except the vertex in $A_{11}$ have degree $d$, and the vertex in $A_{11}$ has degree $d-1$. Therefore, $B_d$ is almost $d$-regular when $d$ is odd. Moreover, each $B_d$ is a (partial) subgraph of a graph that is a blow-up of a cycle of length five, which implies that each $B_d$ is triangle-free. Therefore, we only need to show that $B_d$ is factor-critical. We note that  $B_{4k+1}$  and $B_{4k+3}$ can be obtained from $B_{4k+2}$ and $B_{4k+4}$, respectively, by deleting some edges. Thus, it suffices to show that $B_{4k+1}$ and $B_{4k+3}$ are factor-critical. For every vertex $v$ in $B_d$, we will show that $B_d-v$ has a perfect matching. Due to symmetry, it is enough to examine the cases $v\in A_1\cup A_2\cup A_4$. Since all the examinations are quite similar and straight-forward, we will only show the case $v\in A_1$ and leave the rest to the reader. It is well-known that any regular bipartite graph has a perfect matching. We will show that the vertices in $B_d-v$ can be partitioned into some pairs of subsets so that each pair induces a regular bipartite graph and thus admits a perfect matching. If $v\in A_{11}$, then we can partition the vertices into pairs of subsets as $(A_{22},A_{13})$, $(A_{12},A_{32})$, $(A_{21}\cup A_{23},A_{42})$, $(A_{31}\cup A_{33}, A_{52})$ and $(A_{41},A_{51})$. If $v\in A_{12}\cup A_{13}$, without loss of generality we can assume $v\in A_{12}$. Similarly, we can partition the vertices into pairs of subsets $(A_{11}\cup (A_{12}-v), A_{32})$, $(A_{13},A_{22})$, $(A_{21}\cup A_{23},A_{42})$, $(A_{31}\cup A_{33},A_{52})$ and $(A_{41},A_{51})$. Since $B_d$ is factor-critical, we have  $|V(B_d)|=2\nu(B_d)+1$, and a maximum matching saturates all vertices but one; expressing the number of vertices as a function of $k$ in each one of the four cases, it can be checked that we have $\nu(B_d)= \lfloor5d/4\rfloor$ if $d$ is even, and $\nu(B_d)=\lfloor5(d+1)/4\rfloor$ if $d$ is odd. Lastly,  $|E(B_d)|=d \nu(B_d)+\lfloor d/2\rfloor$ follows from the fact that $B_d$ is factor-critical and (almost) $d$-regular.
\end{proof}

\noindent For any $d\geq2$, let $C_d$ be a (almost) $d$-regular triangle-free factor-critical graph with matching number $Z(d)$. An important consequence of the properties of the graphs $B_d$ shown in Proposition \ref{prop:properties-B(d)} is the following:
\begin{corollary}\label{cor:exist}
For every $d\geq 2$, the value $Z(d)$ and a triangle-free factor-critical (almost) $d$-regular graph $C_d$ with matching number $Z(d)$ exist.
\end{corollary}

\noindent Now, we are ready to show that the matching number of each connected component of a graph $G\in \mathcal{G}_{\vartriangle}(d,m)$ is bounded above by $Z(d)$. Indeed, this additional information on the structure of connected components in an extremal graph will be very useful in both calculating $f_{\vartriangle}(d,m)$ in the rest of this section, and guiding us for future research to complete the remaining open cases.

\begin{lemma}\label{lem:bound-5d-over-4}
	Let $d$ and $m$ be natural numbers with $d\geq 2$, and let $G\in \mathcal{G}_{\vartriangle}(d,m)$. Then, for every connected component $H$ of $G$, we have $\nu(H)\leq Z(d)$. 
\end{lemma}
\begin{proof}
	For a contradiction, let $G\in \mathcal{G}_{\vartriangle}(d,m)$ and $H$ be a connected component of $G$ with $\nu(H)=Z(d)+t$ for some $t\geq 1$. By Corollary \ref{cor:structure-not-star} (ii), we know that $|V(H)|=2\nu(H)+1$. Since $\Delta(H)\leq d$, we have $$|E(H)|\leq \lfloor(2\nu(H)+1)d/2\rfloor=\nu(H)d+\lfloor d/2 \rfloor.$$ On the other hand, let $G_1$ be the graph obtained by taking the disjoint union of $G-H$, the graph $C_d$, and $t$ many $d$-stars. Notice that $G_1$ has more $d$-stars than $G$, so we have $|E(G_1)|<|E(G)|$ by definition of $\mathcal{G}_{\vartriangle}(d,m)$. However, we can write
	\begin{eqnarray*}
		|E(G_1)| &=& |E(G-H)|+(d Z(d)+\lfloor d/2\rfloor)+dt\\
		&=&|E(G-H)|+\nu(H)d+\lfloor d/2\rfloor\\
		&\geq& |E(G-H)|+|E(H)|=|E(G)|,
	\end{eqnarray*}
	which is a contradiction. 
\end{proof}

We can use Lemma \ref{lem:factor-critical-triangle-free} to find the exact value of $Z(d)$ for small values of $d$. These values, on one hand, will allow us to show that $Z(d)\geq d$, thus we can address the case $Z(d)\leq m\leq 2d$ within the case $m>d$, on the other hand, they will be useful while solving the case $d\leq 6$. 

\begin{lemma}\label{lem:smaller-Z(d)-values}
We have $Z(d)=d$ for $d\in\{2,3\}$, and $Z(d)= d+1$ for $d\in\{4,5\}$. Moreover, $Z(d)\geq d+1$ holds for all $d\geq 4$.
\end{lemma}
\begin{proof}
By Lemma \ref{lem:factor-critical-triangle-free}, we have $|E(C_d)|\leq 1+Z(d)^2$ since $C_d$ is factor-critical and triangle-free. Since $C_d$ is (almost) $d$-regular, we get $$|E(C_d)|=\lfloor(2 Z(d)+1)d/2\rfloor=d Z(d)+\lfloor d/2\rfloor.$$ Hence, we obtain $\lfloor d/2\rfloor-1\leq Z(d) (Z(d)-d)$. Since $\lfloor d/2\rfloor-1\geq 0$ for all $d\geq2$, and $\lfloor d/2\rfloor-1\geq 1$ for all $d\geq 4$, we get $Z(d)\geq d$ for all $d\geq2$ and $Z(d)\geq d+1$ for all $d\geq 4$. By Proposition \ref{prop:properties-B(d)}, $B_2$ and $B_4$ are factor-critical and triangle-free graphs with $\nu(B_2)=2$ and $\nu(B_4)=5$, respectively. Also, $B_2$ is $2$-regular and $B_4$ is $4$-regular. Therefore, we get $Z(2)=2$ and $Z(4)=5$. Besides, $A_3$ (see Figure \ref{fig:A_d}) is an almost $3$-regular triangle-free and factor-critical graph with $\nu(A_3)=3$, which shows $Z(3)=3$. Finally, we identified using a computer search the graph $M_5$ given in Figure \ref{fig:M} as the unique triangle-free graph which is both factor-critical with $\nu(M_5)=6$ and almost $5$-regular; this shows $Z(5)=6$. 
\end{proof}

\begin{figure}[h!]
    \centering
	\includegraphics[scale=0.4]{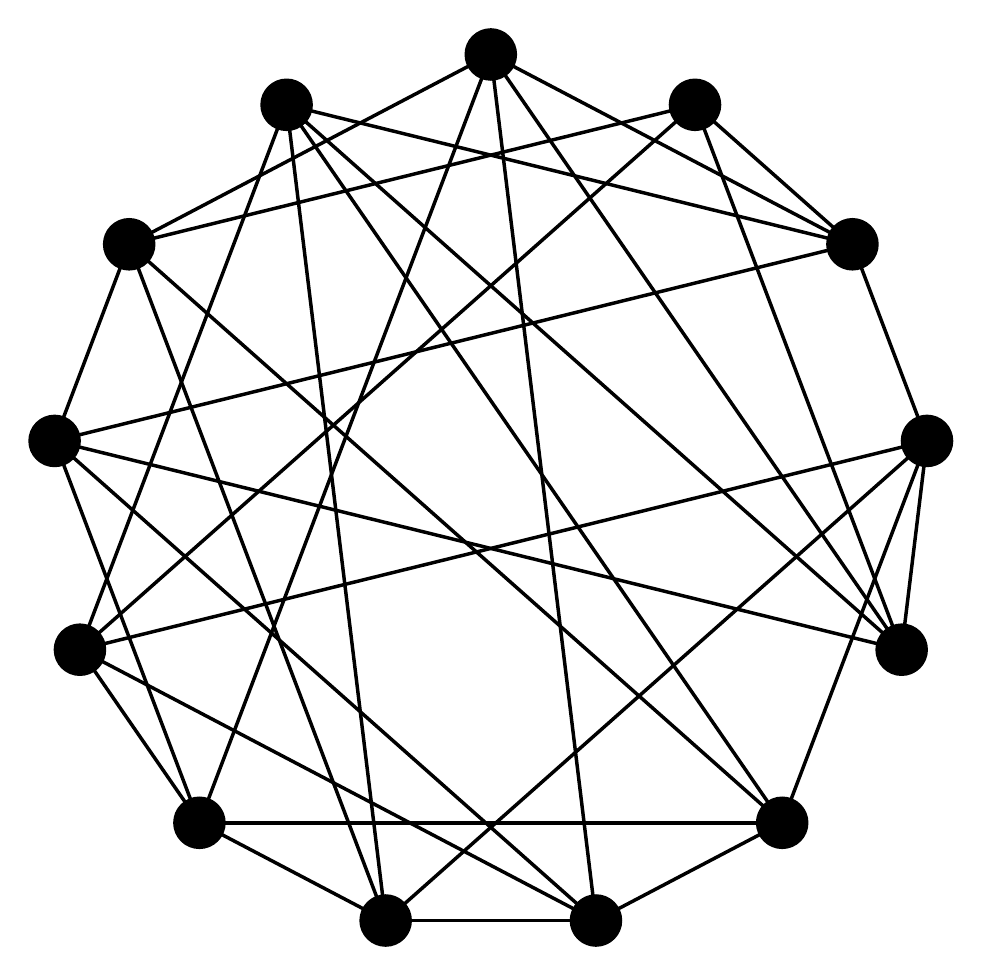}
   \caption{\small{The graph $M_5$.}}\label{fig:M}
      
\end{figure}

\noindent As for larger $d$, Corollary \ref{key} allows us to obtain the exact value of $Z(d)$ for even values of $d$, and to identify a very restricted interval for $Z(d)$ if $d$ is odd.

\begin{lemma}\label{lem:lower-bound-Z(d)}
For $d\geq 2$, if $d$ is even then we have $Z(d)=\lfloor5d/4\rfloor$; if $d$ is odd then we have $\lfloor5(d-1)/4\rfloor \leq Z(d)\leq  \lfloor5(d+1)/4\rfloor $.
\end{lemma}
\begin{proof}
Since factor-critical graphs are non-bipartite, $|V(C_d)|=2 Z(d)+1$ and $\delta(C_d)=2\lfloor d/2\rfloor$, we get $2\lfloor d/2\rfloor\leq \dfrac{2(2 Z(d)+1)}{5}$ by Corollary \ref{key}, which gives $Z(d)\geq \lfloor5d/4\rfloor$ when $d$ is even and $Z(d)\geq \lfloor5(d-1)/4\rfloor$ when $d$ is odd. On the other hand, we have $\nu(B_d)=\lfloor 5d/4\rfloor$ when $d$ is even and $\nu(B_d)=\lfloor 5(d+1)/4\rfloor$ when $d$ is odd. Since $Z(d)\leq \nu(B_d)$, the result follows.
\end{proof}

\noindent Now, by Lemmas \ref{lem:smaller-Z(d)-values} and \ref{lem:lower-bound-Z(d)}, it is clear that $d\leq Z(d) <2d$ for any $d\geq 2$. Now we have the necessary ingredients to give the exact value of $f_{\vartriangle}(d,m)$ for $Z(d) \leq m < 2d$.

\begin{theorem}\label{thm:5d-over-4-case}
With the preceding notation, for $d\geq 2$ and $Z(d)\leq m <2d$, we have $f_{\vartriangle}(d,m)=d  m+\lfloor d/2\rfloor$.
\end{theorem}
\begin{proof}
Let $T$ be the disjoint union of $C_d$ and $m-Z(d)$ many $d$-stars. Clearly we have $\Delta(T)=d$, $\nu(T)=m$ and $$|E(T)|=d Z(d)+\lfloor d/2\rfloor+d(m-Z(d))=dm+\lfloor d/2\rfloor,$$ which shows $f_{\vartriangle}(d,m)\geq dm+\lfloor d/2\rfloor$. Then, let us take $G\in \mathcal{G}_{\vartriangle}(d,Z(d))$. Hence, we have $|E(G)|=f_{\vartriangle}(d,Z(d))\geq dm+\lfloor d/2\rfloor$, so it suffices to show that $|E(G)|\leq dm +\lfloor d/2\rfloor$. Assume $G$ has at least two connected components $H_1$ and $H_2$ that are not $d$-stars. Then, by part (ii) of Corollary \ref{cor:structure-not-star}, we get $2d \leq \nu(H_1)+\nu(H_2)\leq m$, however $m<2d$ by assumption, which is a contradiction. Moreover, if all the connected components of $G$ are $d$-stars, then we would get $|E(G)|=dm$, which contradicts with $|E(G)|\geq dm+\lfloor d/2\rfloor$. Therefore, $G$ has exactly one connected component that is not a $d$-star. Suppose $G$ has $t$ many connected components that are $d$-stars, and let $H$ be the connected component of $G$ that is not a $d$-star. Again, by part (ii) of Corollary \ref{cor:structure-not-star}, we know $|V(H)|=2\nu(H)+1$. On the other hand, since $\Delta(H)\leq d$, we have $$|E(H)|\leq \lfloor(2\nu(H)+1)d/2\rfloor=\nu(H)d+\lfloor d/2\rfloor.$$ Hence, by using $m=t+\nu(H)$, we get $$|E(G)|\leq dt+(\nu(H)d+\lfloor d/2\rfloor)=dm+\lfloor d/2\rfloor,$$ which completes the proof.
\end{proof}


In the sequel, we will reformulate our problem in a slightly different way to calculate $f_{\vartriangle}(d,m)$ for the case where $m>d$ and $d\leq 6$. This reformulation will be revisited in Section \ref{sec:discussion} to suggest an integer programming formulation and discuss future research directions for the remaining open cases in Section \ref{sec:discussion}.

Let us take a graph $G\in\mathcal{G}_{\vartriangle}(d,m)$ for some natural numbers $1\leq d\leq m$. For any connected component of $G$ that is not a $d$-star, say $H$, we know $d\leq \nu(H)\leq Z(d)$ and $|E(H)|=f_{\vartriangle}(d,\nu(H))$ by Corollary \ref{cor:structure-not-star} and Lemma \ref{lem:bound-5d-over-4}. Then, let $x_i$ be the number of connected components of $G$ whose matching number is $i$ where $d\leq i\leq Z(d)$. Clearly, we have $\displaystyle{\sum_{i=d}^{Z(d)}ix_i\leq m}$, and $G$ has $\displaystyle{m-\sum_{i=d}^{Z(d)}ix_i}$ many connected components that are $d$-stars. Therefore, we can write the number of edges in $G$ in terms of $x_i$'s as follows:
\begin{eqnarray*}
f_{\vartriangle}(d,m)=|E(G)|&=&d\Big(m-\sum_{i=d}^{Z(d)}ix_i\Big)+\sum_{i=d}^{Z(d)}f_{\vartriangle}(d,i)x_i\\
&=&dm-\sum_{i=d}^{Z(d)}dix_i+\sum_{i=d}^{Z(d)}f_{\vartriangle}(d,i)x_i\\
&=&dm+\sum_{i=d}^{Z(d)}(f_{\vartriangle}(d,i)-di)x_i.
\end{eqnarray*}

As a result, for a fixed $d$, we can determine the value of $f_{\vartriangle}(d,m)$ for all natural numbers $m$ by finding the values of $f_{\vartriangle}(d,i)$ and corresponding $x_i$ values only for $d\leq i\leq Z(d)$. For a simpler notation, let us define
\begin{eqnarray*}
\mathcal{F}(d,m)&:=&\Big\{\big(x_d,x_{d+1},\ldots,x_{Z(d)}\big):x_i\in\mathbb{Z}_{\geq 0}\text{ for }d\leq i\leq Z(d),\sum_{i=d}^{Z(d)}ix_i\leq m\Big\},\\
g_{\vartriangle}(d,i)&:=&f_{\vartriangle}(d,i)-di\text{ for any }i.
\end{eqnarray*}

Observe that we have 
$g_{\vartriangle}(d,d)=1$ for $d\geq 2$ by Theorem \ref{thm:case-d-d}. Also, we get $g_{\vartriangle}(d,Z(d))=\lfloor d/2\rfloor$ for $d\geq 2$ by Theorem \ref{thm:5d-over-4-case}. Now, we state the discussion above as a lemma since it will be used in the calculations for the cases $2\leq d\leq 6$.

\begin{lemma}\label{lem:MAIN-RECURSION}
For all natural numbers $1\leq d\leq m$, we have $$g_{\vartriangle}(d,m)=\max_{(x_d,\ldots,x_{Z(d)})\in\mathcal{F}(d,m)}\sum_{i=d}^{Z(d)}g_{\vartriangle}(d,i)x_i.$$
\end{lemma}

Notice that we have $Z(d)=d$ for $d\in\{2,3\}$ and $Z(d)=d+1$ for $d\in\{4,5,6\}$ by Lemmas \ref{lem:smaller-Z(d)-values} and \ref{lem:lower-bound-Z(d)}. Therefore, we have a simple expression for $\mathcal{F}(d,m)$ for $2\leq d\leq 6$, which helps to find the exact value of $f_{\vartriangle}(d,m)$.
\begin{theorem}\label{thm:d123}
With the preceding notation, for $m>d$ and $d\in\{2,3\}$, $$f_{\vartriangle}(d,m)=dm+\lfloor m/d\rfloor \lfloor d/2\rfloor.$$
\end{theorem} 
\begin{proof}
Since $d=Z(d)$ for  $d\in\{2,3\}$ by Lemma \ref{lem:smaller-Z(d)-values}, $\mathcal{F}(d,m)$ contains only $1$-dimensional elements, so we get $\mathcal{F}(d,m)=\{x_d\in\mathbb{Z}_{\geq 0}:dx_d\leq m\}=\{ x_d\in\mathbb{Z}_{\geq 0}:x_d\leq \lfloor m/d\rfloor\}.$ Then, we have $$g_{\vartriangle}(d,m)=\max_{x_d\in\mathcal{F}(d,m)}g_{\vartriangle}(d,d)x_d=g_{\vartriangle}(d,d) \lfloor m/d\rfloor$$ by Lemma \ref{lem:MAIN-RECURSION}. As we have already inferred from Theorem \ref{thm:5d-over-4-case} that $g_{\vartriangle}(d,d)=g_{\vartriangle}(d,Z(d))=\lfloor d/2\rfloor$, we get $$f_{\vartriangle}(d,m)-dm=g_{\vartriangle}(d,m)=\lfloor d/2\rfloor  \lfloor m/d\rfloor,$$ and the result follows.
\end{proof}

\begin{theorem}\label{thm:d4567}
With the preceding notation, for $m>d$ and $d\in\{4,5,6\}$ we have, $$f_{\vartriangle}(d,m)=\begin{cases}
	1+dm+\lfloor d/2\rfloor \lfloor m/(d+1)\rfloor,&\text{if }m+1\text{ is divisible by }d+1,\\
	dm+\lfloor d/2\rfloor \lfloor m/(d+1)\rfloor,&\text{otherwise}.\\
\end{cases}$$
\end{theorem}
\begin{proof}
Let $d\in\{4,5,6\}$. Since $Z(d)=d+1$ by Lemmas \ref{lem:smaller-Z(d)-values} and \ref{lem:lower-bound-Z(d)}, $\mathcal{F}(d,m)$ contains 2-dimensional elements, and we get $$\mathcal{F}(d,m)=\{(x_d,x_{d+1})\in \mathbb{Z}_{\geq0}^2:x_d\leq (m-(d+1)x_{d+1})/d\}.$$ 
On the other hand,  we have $$g_{\vartriangle}(d,m)=\max_{(x_d,x_{d+1})\in\mathcal{F}(d,m)}g_{\vartriangle}(d,d)x_d+g_{\vartriangle}(d,d+1)x_{d+1}$$ by Lemma \ref{lem:MAIN-RECURSION}. Since $g(d,d)=1$ and $g(d,d+1)=g(d,Z(d))=\lfloor d/2\rfloor$, we can write
\begin{eqnarray*}
g_{\vartriangle}(d,m)&=&\max_{(x_d,x_{d+1})\in\mathcal{F}(d,m)}(x_d+\lfloor d/2\rfloor x_{d+1})\\
&=&\max_{0\leq x_{d+1}\leq  m/(d+1)} \lfloor(m-(d+1)x_{d+1})/d\rfloor+\lfloor d/2\rfloor x_{d+1}\\
&=&\max_{0\leq x_{d+1}\leq  m/(d+1)} \lfloor (m-x_{d+1})/d\rfloor + \lfloor (d-2)/2\rfloor x_{d+1}.
\end{eqnarray*}
Since $\lfloor (d-2)/2\rfloor\geq 1$, the quantity $\lfloor (m-x_{d+1})/d\rfloor + \lfloor (d-2)/2\rfloor x_{d+1}$ increases with respect to $x_{d+1}$. Therefore, $g_{\vartriangle}(d,m)$ is obtained by assigning $x_{d+1}=\lfloor m/(d+1)\rfloor$ which is the maximum possible value for $x_{d+1}$. Then, by writing $m=(d+1)k+r$ for some $k,r\in\mathbb{Z}$ where $k=\lfloor m/(d+1)\rfloor$ and $0\leq r\leq d$, we see that $(x_d,k)\in \mathcal{F}(d,m) $ implies $x_d\leq 1$ if $r=d$ and $x_d\leq 0$ otherwise as $x_d\leq (m-(d+1)x_{d+1})/d$. Note that $r=d$ is equivalent to the case that $m+1$ is divisible by $d+1$. Therefore, the value $g_{\vartriangle}(d,m)$ is attained at $x_d=1$, $x_{d+1}=k$ if $m+1$ is divisible by $d+1$, and it is attained at $x_d=0$, $x_{d+1}=k$ otherwise. As a result, we find
 $$g_{\vartriangle}(d,m)=\begin{cases}
1+\lfloor d/2\rfloor \lfloor m/(d+1)\rfloor,&\text{if }m+1\text{ is divisible by }d+1,\\
\lfloor d/2\rfloor \lfloor m/(d+1)\rfloor,&\text{otherwise},\\
\end{cases}$$ so the result follows.
\end{proof}

\section{Main Result}\label{sec:conc}

 We determined the value of $f_{\vartriangle}(d,m)$ for all the cases with $d\geq m$ (Theorems \ref{thm:case-d-greater-m} and \ref{thm:case-d-d}), and for the cases $d<m$ with either  $Z(d) \leq m <2d$ (Theorem \ref{thm:5d-over-4-case}) or $d\leq 6$ (Theorems \ref{thm:d123} and \ref{thm:d4567}). It is possible to summarize those findings in a single formula that we state as our main result. Recall that $C_d$ is a (almost) $d$-regular triangle-free factor-critical graph with matching number $Z(d)$ whose existence is guaranteed by Proposition \ref{prop:properties-B(d)} and Corollary \ref{cor:exist}.

\begin{theorem}\label{thm:ALL-IN-ONE}
Let $d$ and $m$ be natural numbers with $d\geq 2$, and let $k$ and $r$ be non-negative integers such that $m=k Z(d)+r$ with $0\leq r<Z(d)$. Then, for all the cases with $d\geq m$, and for the cases $d<m$ with either $d\leq 6$ or $Z(d)\leq m < 2d$, we have
\begin{equation}\label{eqn:single-formula}
f_{\vartriangle}(d,m)=\begin{cases}
	dm+k\lfloor d/2\rfloor&\text{if }r<d,\\
	dm+k\lfloor d/2\rfloor+r-d+1&\text{if }r\geq d,
\end{cases}\tag{*}
\end{equation}
where a graph in $\mathcal{G}_{\vartriangle}(d,m)$ can be constructed as the disjoint union of $k$ copies of $C_d$ and
\begin{itemize}
\item[(i)] $A_d$ if $r\geq d$,
\item[(ii)] $r$ copies of $d$-stars if $r<d$.
\end{itemize}
\end{theorem}	
\begin{proof}
Let $m=k Z(d)+r$ for some non-negative integers $k$ and $r$ with $0\leq r<Z(d)$. If $m<d$, then we find $k=0$ and $r=m$ since $Z(d)\geq d$ by Lemma \ref{lem:smaller-Z(d)-values}, so \eqref{eqn:single-formula} holds by Theorem \ref{thm:case-d-greater-m}. If $d=m\in\{2,3\}$, then we get $k=1$ and $r=0$ since $Z(d)=d$ by Lemma \ref{lem:smaller-Z(d)-values}. Since $\lfloor d/2\rfloor=1$, \eqref{eqn:single-formula} holds by Theorem \ref{thm:case-d-d}. If $d=m\geq 4$, then we get $k=0$ and $r=d$ since $Z(d)=d+1$ by Lemma \ref{lem:smaller-Z(d)-values}. Since $r-d+1=1$, \eqref{eqn:single-formula} holds by Theorem \ref{thm:case-d-d}. Suppose $d<m$. If $d\in\{2,3\}$, then we get $Z(d)=d$, which implies $k=\lfloor m/d\rfloor$ and $r<d$, so \eqref{eqn:single-formula} holds by Theorem \ref{thm:d123}. If $d\in\{4,5,6\}$, then we get $Z(d)=d+1$, which implies $k=\lfloor m/(d+1)\rfloor$. Moreover, we find $r=d$ if $m+1$ is divisible by $k+1$ and $r<d$ otherwise. Therefore, if $m+1$ is divisible by $k+1$ then $r-d+1=1$, so \eqref{eqn:single-formula} holds by Theorem \ref{thm:d4567}. Finally, if $Z(d) \leq m <2d$, then since $d\leq Z(d)$ we have $m<2Z(d)$, thus $k=1$ and $r<d$, then \eqref{eqn:single-formula} holds by Theorem \ref{thm:5d-over-4-case}.  
\end{proof}

Now, we can report the difference between $f_{\mathcal{GEN}}(d,m)$ and  $f_{\vartriangle}(d,m)$ based on our findings. Theorems \ref{thm:general} and \ref{thm:ALL-IN-ONE} give $$h_{\vartriangle}(d,m):=f_{\mathcal{GEN}}(d,m)-f_{\vartriangle}(d,m)=\left\lfloor \frac{d}{2} \right\rfloor  \Bigg(\left\lfloor \frac{m}{\lceil \frac{d}{2} \rceil} \right\rfloor-k\Bigg)-(r-d+1),$$ where $m=k Z(d)+r$ with $0\leq r<Z(d)$ provided that $d$ and $m$ satisfy one of the following conditions:
	\begin{itemize}
	\item[(i)] $d\geq m$,
	\item[(ii)] $d<m$ and $d\leq 6$,
	\item[(iii)] $d<m$ and $Z(d)\leq m < 2d$.
	\end{itemize}
The difference $h_{\vartriangle}(d,m)$ corresponds to the number of edges that we loose in the triangle-free case as compared to the general one. Thus, we get the following where we observe that apart from two simple cases (where $m<\lfloor d/2\rfloor$ or  $1=d< m$), we loose edges (with respect to the general case) by restricting the extremal graphs to be triangle-free: \[h_{\vartriangle}(d,m)=\begin{cases}
0,&\text{if }m<\lfloor d/2\rfloor,\\
\lfloor d/2\rfloor,&\text{if }\lfloor d/2\rfloor\leq m<d\\
d,&\text{if }d=m\text{ and $d$ is even},\\
\lfloor d/2\rfloor,&\text{if }d=m\text{ and $d$ is odd},\\
0,&\text{if $1=d< m$,}\\
m-\lfloor m/2\rfloor,&\text{if $2=d<m$,}\\
\lfloor m/2\rfloor-\lfloor m/3\rfloor,&\text{if $3=d<m$,}\\
2\lfloor m/2\rfloor-2\lfloor m/5\rfloor,&\text{if $4=d<m$ and $m+1$ is not divisible by 5,}\\
2\lfloor m/2\rfloor-2\lfloor m/5\rfloor-1,&\text{if $4=d<m$ and $m+1$ is divisible by 5,}\\
2\lfloor m/3\rfloor-2\lfloor m/6\rfloor,&\text{if $5=d<m$ and $m+1$ is not divisible by 6,}\\
2\lfloor m/3\rfloor-2\lfloor m/6\rfloor-1,&\text{if $5=d<m$ and $m+1$ is divisible by 6,}\\
3\lfloor m/3\rfloor-3\lfloor m/7\rfloor,&\text{if $6=d<m$  and $m+1$ is not divisible by 7,}\\
3\lfloor m/3\rfloor-3\lfloor m/7\rfloor-1,&\text{if $6=d<m$ and $m+1$ is divisible by 7,}\\
\lfloor d/2\rfloor,&\text{if $d\geq 7$ and $Z(d)\leq m<3\lceil d/2\rceil$,}\\
2\lfloor d/2 \rfloor,&\text{if $d\geq 7$ and $3\lceil d/2\rceil \leq m<2d$.}

\end{cases}\]

In light of Theorem \ref{thm:ALL-IN-ONE}, the remaining open cases are for $7\leq d< m$, and either $m<Z(d)$ or $m\geq 2d$. In what follows, we will discuss further formulations to solve these remaining cases and suggest some conjectures.

\section{An integer programming formulation and further discussions}\label{sec:discussion}
To solve the open cases, namely for natural numbers $m$ and $d$ such that $7\leq d< m$ with either $m<Z(d)$ or $m\geq 2d$, we develop an integer programming formulation based on our earlier observations. In Conjecture \ref{con:Turan-generalization}, we provide all the parameters involved in this formulation, which is already a challenging problem. Then, under the assumption that Conjecture \ref{con:Turan-generalization} holds, we show that our integer program admits an optimal solution with a special structure. This, in turn, allows us to formulate in Conjecture \ref{con:general-formula-wrt-Z(d)} that Theorem \ref{thm:ALL-IN-ONE} is valid for all $m$ and $d$. Lastly, we also conjecture unknown values of $Z(d)$ (see Lemma \ref{lem:lower-bound-Z(d)}), which plays a crucial role in the solution of our problem. We conclude our paper with a reformulation of our problem as a variant of the well-known extremal problem addressed in Turan's Theorem.

By Corollary \ref{cor:structure-not-star} and Lemma \ref{lem:bound-5d-over-4}, there is an edge-extremal graph $G\in \mathcal{G}_{\vartriangle} (d,m)$ whose components are either $d$-stars, or edge-extremal factor-critical triangle-free graphs $H$ where $d\leq \nu(H)\leq Z(d)$. In other words, by letting $x_i$ to be the number of connected components of $G$ whose matching number is $i$, we have (as expressed in Lemma  \ref{lem:MAIN-RECURSION} in terms of $g_{\vartriangle}(d,m)$):

$$f_{\vartriangle}(d,m)= d\Big(m-\sum_{i=d}^{Z(d)}ix_i\Big)+\sum_{i=d}^{Z(d)}f_{\vartriangle}(d,i)x_i
= dm+\sum_{i=d}^{Z(d)}(f_{\vartriangle}(d,i)-di)x_i.$$

It follows that, for a fixed $d$, the value of $f_{\vartriangle}(d,m)$ can be determined for all natural numbers $m$ by finding the values of $f_{\vartriangle}(d,i)$ and corresponding $x_i$ values only for $d\leq i\leq Z(d)$. Accordingly, $f_{\vartriangle}(d,m)$ can be computed as the optimal value of the following integer programming:\\

\begin{center}
\flushleft \textbf{Model 1: }
	$$\displaystyle{\max \,\,  dm+\sum_{i=d}^{Z(d)}(f_{\vartriangle}(d,i)-di)x_i}$$	
	$$\displaystyle{ {\mbox{subject to }}\sum_{i=d}^{Z(d)}ix_i\leq m}$$
	$$x_i\geq 0, x_i\in\mathbb{Z}$$
\end{center}

This formulation can be seen as a bounded knapsack problem where there is a bounded number of items of each type. The utilities of the items are $(f_{\vartriangle}(d,i)-di)$  for $d\leq i\leq Z(d)$ and the volumes of the items range from $d$ to $Z(d)$ which is yet unknown if $d$ is odd (see Lemma \ref{lem:lower-bound-Z(d)}).

Recall that we have $f_{\vartriangle}(d,d)=d^2+1$ for $d\geq 2$ by Theorem \ref{thm:case-d-d}, and $f_{\vartriangle}(d,Z(d))=dZ(d)+\lfloor d/2\rfloor$ for $d\geq 2$ by Theorem \ref{thm:5d-over-4-case}. It remains to compute $f_{\vartriangle}(d,i)$ for $d < i < Z(d)$. We suggest Conjecture \ref{con:Turan-generalization} for the value of $f_{\vartriangle}(d,i)$ for $d < i < Z(d)$, which in turn, allows us to conjecture that the formula for $f_{\vartriangle}(d,m)$ in Theorem \ref{thm:ALL-IN-ONE} can be extended to all the remaining cases as well (in Conjecture \ref{con:general-formula-wrt-Z(d)}). Lastly, bearing in mind that the formula giving the value of $f_{\vartriangle}(d,m)$ can only be computed if $Z(d)$ is known; we suggest Conjecture \ref{con:open-Z(d)-values} to settle the values of $Z(d)$ for odd $d\geq 21$, which is left open (see Lemma \ref{lem:lower-bound-Z(d)}).

In what follows, we conjecture that for $d< i<Z(d)$,  $f_{\vartriangle}(d,i)$ follows the same trend as what we identified in other cases.

\begin{conjecture}\label{con:Turan-generalization}
Theorem \ref{thm:ALL-IN-ONE} holds also for $7\leq d<m<Z(d)$. In other words, for $7\leq d< i<Z(d)$, we have $$f_{\vartriangle}(d,i)=di+i-d+1.$$
\end{conjecture}

\noindent If Conjecture \ref{con:Turan-generalization} holds, then we get  $f_{\vartriangle}(d,i)-di=i-d+1$ for $2\leq d\leq i<Z(d)$. Since $f_{\vartriangle}(d,Z(d))=dZ(d)+\lfloor d/2 \rfloor$ and the constant term $dm$ in the objective function does not effect the optimal solution, Model 1 is equivalent to solve the following optimization problem: 

\begin{center}
\flushleft \textbf{Model 2:} 
	$$\displaystyle{\max \,\, \lfloor d/2\rfloor x_{Z(d)}+\sum_{i=d}^{Z(d)-1}(i-d+1)x_i}$$	
	$$\displaystyle{ {\mbox{subject to }}\sum_{i=d}^{Z(d)}ix_i\leq m}$$
	$$x_i\geq 0, x_i\in\mathbb{Z}$$
\end{center}

\noindent We claim that if Conjecture \ref{con:Turan-generalization} holds, then Model 2 admits an optimal solution with nice properties.  First, we need a direct consequence of Conjecture \ref{con:Turan-generalization}.

\begin{proposition}\label{prop:Conjecture-1-Z(7)}
If Conjecture \ref{con:Turan-generalization} is true, then we have $Z(7)\in\{8,9\}$.
\end{proposition}
\begin{proof}
We claim that $Z(7)$ is the smallest natural number $k$ satisfying $f_{\vartriangle}(7,k)=7k+3$. Indeed, we know that $f_{\vartriangle}(7,Z(d))=7Z(d)+3$ by Theorem \ref{thm:5d-over-4-case}. Now, suppose $f_{\vartriangle}(7,k)=7k+3$ for some $k<Z(7)$. Then, we can take a graph $G\in \mathcal{G}_{\vartriangle}(7,k)$ with $7k+3$ edges. By Corollary \ref{cor:structure-not-star}, we know that each connected component $H$ of $G$ that is not a $7$-star is factor-critical with $\nu(H)\geq 7$. Since $\nu(H)\leq k<Z(7)$, it follows that $H$ is not almost $7$-regular, so $E(H)\leq ((2\nu(H)+1)7-3)/2=7\nu(H)+2$. By summing the number of edges over all connected components of $G$, we find $7k+3=|E(G)|\leq 7k+2$, which is a contradiction. As a result, we have $f_{\vartriangle}(7,k)<7k+3$ for $k<Z(d)$. Now, if Conjecture \ref{con:Turan-generalization} holds and $Z(7)\geq 10$, the above discussion implies that $f_{\vartriangle}(7,9)=66< 7\times 9+3$, which is a contradiction. Since $Z(7)\geq 8$ by Lemma \ref{lem:smaller-Z(d)-values}, the result follows. 

\end{proof}
\noindent Now, we are ready to discuss the optimal solution to Model 2 that admits nice properties.

\begin{proposition}\label{prop:Conjecture-1-Implies}
If Conjecture \ref{con:Turan-generalization} is true, then for $7\leq d<m<Z(d)$, Model 2 admits an optimal solution with $\displaystyle{\sum_{i=d}^{Z(d)-1}x_i\leq 1}$, that is where $x_{Z(d)}$ is maximized, and there is at most one other $x_i$ which is 1 (all the rest being zero). 
\end{proposition}
\begin{proof}
Let $(x_{d},x_{d+1},\ldots,x_{Z(d)-1},x_{Z(d)})$ be an optimal solution for Model 2 with optimal value $opt$ and such that $x_d+x_{Z(d)}$ is maximal. We first show that $x_j\leq 1$ for all $d<j<Z(d)$. Assume the contrary. 

If $j\geq \dfrac{d+Z(d)}{2}$, then we can decrease $x_j$ by two, and increase each of $x_{2j-Z(d)}$ and $x_{Z(d)}$ by one, which gives another feasible solution for Model 2 with objective value $opt+(2j-Z(d)-d+1)+\lfloor d/2\rfloor-2(j-d+1)=opt+\lfloor(3d-2)/2\rfloor-Z(d)$. Since we increased $x_d+x_{Z(d)}$ by at least one, the new solution is not optimal by assumption, thus we have $\lfloor(3d-2)/2\rfloor\leq Z(d)-1$. By Lemma \ref{lem:lower-bound-Z(d)}, we obtain $\lfloor5(d+1)/4\rfloor\geq Z(d)\geq \lfloor 3d/2\rfloor$, which is a contradiction for $d\geq 8$. On the other hand, for $d=7$, Proposition \ref{prop:Conjecture-1-Z(7)} implies that $Z(7)\leq9$, which contradicts with $\lfloor(3d-2)/2\rfloor\leq Z(d)-1$. As a result, if $j\geq \dfrac{d+Z(d)}{2}$, we have $x_j\leq 1$. \\

\noindent If $j<\dfrac{d+Z(d)}{2}$, then we can decrease $x_j$ by two, and increase each one of $x_d$ and $x_{2j-d}$ by one, which would give a feasible solution with the objective value $opt+1+(2j-d+1)-2(j-d+1)=opt$. Since we increased $x_d+x_{Z(d)}$ by at least one, we get a contradiction. Therefore, we have $x_j\leq 1$ for all $d<j<Z(d)$.\\ 

\noindent Now, suppose $x_a=x_b=1$ for some $d<a<b<Z(d)$. If $a+b\geq Z(d)+d$, then let us decrease $x_a$ and $x_b$ by one, and increase $x_{a+b-Z(d)}$ and $x_{Z(d)}$ by one. By this way, we get a feasible solution with the objective value $$opt-(a-d+1)-(b-d+1)+(a+b-Z(d)-d+1)+\lfloor d/2\rfloor.$$ Since we increased $x_d+x_{Z(d)}$ by at least one, we should have $\lfloor(3d-2)/2\rfloor\leq Z(d)-1$ by the assumption. As similar to previous cases, this inequality does not hold for $d\geq 7$. If $a+b<d+Z(d)$, then let us decrease $x_a$ and $x_b$ by one, and increase $x_d$ and $x_{a+b-d}$ by one. By this way, we get a feasible solution with the objective value $$opt-(a-d+1)-(b-d+1)+1+((a+b-d)-d+1)=opt.$$ Since we increased $x_d+x_{Z(d)}$ by at least one, we get a contradiction. Therefore, we can say that $x_j\geq1$ holds for at most one $j$ value with $d<j<Z(d)$.\\

\noindent Now, if $x_d\geq 2$, let us decrease $x_d$ by two, and increase $x_{Z(d)}$ by 1. This yields a feasible solution with the objective value $opt-2+\lfloor d/2\rfloor>opt$, which is a contradiction. Hence, we have $x_d\leq 1$. The only remaining case is $x_d=1$ and there is exactly one $j$ value with $x_j=1$, $d<j<Z(d)$.\\

\noindent Let us decrease $x_d$ and $x_j$ by one, and increase $x_{Z(d)}$ by one. Then, we would get a feasible solution with the optimal value $$opt-1-(j-d+1)+\lfloor d/2\rfloor=opt+\lfloor(3d-4)/2\rfloor-j.$$ If $\lfloor(3d-4)/2\rfloor=j$, then we can obtain the same optimal value with $x_d=x_j=0$, so the result follows. Thus, we are done if we show the inequality $\lfloor(3d-4)/2\rfloor\geq j$ for $d\geq 7$. For $d=7$, note that $Z(7)\in\{8,9\}$ by Proposition \ref{prop:Conjecture-1-Z(7)}. If $Z(7)=8$, then there are no $j$ values with $7<j<Z(7)$, so we are done. If $Z(7)=9$, then, we get $j=8$ and so the equality is satisfied. For $d=8$, we know $Z(8)=10$ by Lemma \ref{lem:lower-bound-Z(d)}, which gives $j\leq 9$ and so $\lfloor(3d-4)/2\rfloor-j>0$. For $d\geq 9$, by using Lemma \ref{lem:lower-bound-Z(d)}, we have $$\lfloor(3d-4)/2\rfloor\geq\lfloor(5d+1)/4\rfloor\geq Z(d)-1\geq j,$$ which completes the proof.
\end{proof}

\noindent Proposition \ref{prop:Conjecture-1-Implies} can be interpreted as follows: under the assumption that Conjecture \ref{con:Turan-generalization} holds, one can reach $f_{\vartriangle}(d,m)$ edges by taking the graph $C_d$ as much as possible and adding either one more graph that is extremal for $f_{\vartriangle}(d,r)$ or $r$ many stars, depending on $r\geq d$ where $r$ is the remainder of $m$ when divided by $Z(d)$. Notice that this is exactly how we construct an extremal graph in Theorem \ref{thm:ALL-IN-ONE}. Therefore, the formula in Theorem \ref{thm:ALL-IN-ONE} would be valid for all integers $d$ and $m$ if Conjecture \ref{con:Turan-generalization} is true:

\begin{conjecture}\label{con:general-formula-wrt-Z(d)}
Let $m=k Z(d)+r$ for some $0\leq r<Z(d)$. Then, we have $$f_{\vartriangle}(d,m)=\begin{cases}
	dm+k\lfloor d/2\rfloor&\text{if }r<d\\
	dm+k\lfloor d/2\rfloor+r-d+1&\text{if }r\geq d.
\end{cases}$$
\end{conjecture}

\noindent Our next conjecture is about the value of $Z(d)$ which plays a crucial role in the computation of $f_{\vartriangle}(d,m)$ and the construction of extremal graphs. Recall that Lemma \ref{lem:smaller-Z(d)-values} together with Lemma \ref{lem:lower-bound-Z(d)} give $Z(d)=\lfloor 5d/4\rfloor$ if $d$ is even or $d\in\{3,5\}$; moreover there is a narrow interval for possible values of $Z(d)$ in the remaining cases (that is $d\geq 7$ odd). In \cite{Haggkvist1}, it is  stated that every triangle-free graph $G$ with $\delta(G)>3| V(G)|/8$ is a subgraph of a blow-up of the cycle of length five. For odd values of $d$, since the graphs $C_d$ realizing $Z(d)$ are triangle-free, almost regular, and factor-critical (by the definition of $Z(d)$),  we have $\delta(C_d)\geq d-2$ and $|V(C_d)|=2Z(d)+1\leq 2\lfloor 5(d+1)/4\rfloor+1$ by Lemma \ref{lem:lower-bound-Z(d)}. Since $d-2> 3(2\lfloor 5(d+1)/4\rfloor+1)/8$ for all but a few small values of $d$, this result implies that these $C_d$ graphs should be blow-up graphs of the cycle of length 5 provided that $d$ is sufficiently large. We could show for some cases that if $Z(d) < \lfloor5(d+1)/4\rfloor$ then it is not possible to construct $C_d$ which is the blow-up of a cycle of length 5 and (almost) regular with degree $d$; thus $Z(d)=\lfloor5(d+1)/4\rfloor$ for these cases. We believe that this also holds for the remaining cases if $d$ is large enough, which we formulate as a conjecture:
\begin{conjecture}\label{con:open-Z(d)-values}
For $d\geq 21$ and odd, we have $Z(d)=\lfloor5(d+1)/4\rfloor$.
\end{conjecture}

Last but not least, let us reformulate the computation of $f_{\vartriangle}(d,i)$ for $d \leq i \leq Z(d)$ as a generalized version of Erdős-Stone's Theorem which has been described as a fundamental theorem of extremal graph theory (see \cite{bollobas}). The extremal number $\text{ex}(n,H)$ is defined as the maximum number of edges in a graph on $n$ vertices not containing a subgraph isomorphic to $H$. Note that the classical Turan's Theorem \cite{turan} addresses the answer for $\text{ex}(n,K_r)$. In our case, we seek for the maximum number of edges in a triangle-free graph whose maximum degree is also bounded by some parameter. Indeed, by Corollary \ref{cor:structure-not-star} and Lemma \ref{lem:bound-5d-over-4}, for a graph $G\in \mathcal{G}_{\vartriangle} (d,m)$, every connected component of $G$ which is not a $d$-star is a factor-critical edge-extremal triangle-free graph with $f_{\vartriangle}(d,i)$ edges, thus with $2i+1$ vertices, where  $d\leq i\leq Z(d)$. Hence, by forbidding not a single graph $H$ but  any graph in a family $\mathcal{F}$ in the Erdős-Stone's Theorem, we can write $f_{\vartriangle}(d,i)=\text{ex}(2i+1,\{K_3,K_{1,d}\})$ for $d\leq i\leq Z(d)$. It follows that we have reduced our original problem of determining the maximum number of edges in a triangle-free graph with degree and matching number bounds into determining $\text{ex}(2i+1,\{K_3,K_{1,d}\})$ for $d\leq i\leq Z(d)$. Let us conclude by noting that Erdős-Stone's Theorem investigates the asymptotic behavior of $\text{ex}(n,\mathcal{F})$ whereas we seek for the exact value in the particular case $\text{ex}(2i+1,\{K_3,K_{1,d}\})$.

\end{document}